\documentclass[11pt]{amsart}
\usepackage[utf8]{inputenc}
\usepackage[english]{babel}
\usepackage{amscd,amsmath,latexsym,amsthm,amsfonts,amssymb,mathrsfs,graphicx,color,geometry}
\usepackage{hyperref}
\hypersetup{
	colorlinks=true,
	linkcolor=red,
	citecolor=green
}

\numberwithin{equation}{section}
\newtheorem{theorem}{Theorem}[section]
\newtheorem{proposition}[theorem]{Proposition}
\newtheorem{corollary}[theorem]{Corollary}

\theoremstyle{definition}

\newtheorem{remark}[theorem]{Remark}
\newtheorem{lemma}[theorem]{Lemma}
\newtheorem*{ack}{Acknowledgments}



\begin{document}

\title[Maximal Spaceability in Quasi-Banach Spaces]{Maximal Spaceability of the set of non-absolutely summing operators in quasi-Banach spaces}
\author{Daniel Ara\'ujo}
\address{Department of Mathematics, Federal Rural University of Pernambuco, Recife, Brazil}
\email{daniel.tomaz@ufrpe.br}


\begin{abstract}
For every $0<p<1$, we establish maximal spaceability of the bounded operators on $\ell_p$ which fail to be $(r,s)$-absolutely summing for all $1\le s\le r<\infty$, using the coordinate-disjointization construction and Maddox's theorem. We give detailed isometric estimates and prove a corresponding maximal spaceability theorem inside the compact operators. The compact result is based on an explicit diagonal multiplier and a single weakly $1$-summable sequence witnessing the simultaneous failure of all admissible absolute summability conditions. A diagonal characterization, a logarithmic compact example, and extensions to operators with values in quasi-Banach spaces containing a copy of $\ell_p$ are also obtained. The constructed subspaces are isomorphic to $\ell_p$ and have maximal algebraic dimension whenever the ambient operator spaces have dimension $\mathfrak c$.
\end{abstract}

\maketitle

\section{Introduction and preliminaries}\label{s1}

The study of linear structures inside nonlinear sets has become a central topic in modern analysis. In particular, the concepts of lineability and spaceability, introduced by Gurariy in \cite{Gurariy1966} and further developed by Aron, Gurariy and Seoane-Sepúlveda in \cite{Aron}, provide a powerful framework to detect large algebraic and topological structures within sets that, at first glance, may appear highly irregular. Since then, this theory has experienced significant development, with deep connections to several areas such as Functional Analysis, Operator Theory, Complex Analysis, Linear Dynamics, Measure Theory, among others (see, for instance, \cite{arb,bag,Bernal,BPS,bot,carmona,conejero}) and the references therein. A comprehensive account of the subject can be found in the monograph \cite{book}. Additional examples involving dimension constraints and hypercyclic subspaces can be found in \cite{maximal,rolewicz}.

In the context of Operator Theory, the investigation of lineability and spaceability phenomena for classes of bounded linear operators has attracted considerable attention. In particular, the structure of non-absolutely summing operators has been extensively studied in Banach spaces, where several results ensure the existence of large linear subspaces under suitable geometric assumptions (see \cite{Pellegrino1,Fakhar,puglisi}). Notably, Kitson and Timoney \cite{KT} established important spaceability results relying on properties such as local convexity. However, the quasi-Banach setting presents a fundamentally different landscape. Spaces such as $\ell_p$ for $0<p<1$ lack local convexity, and many classical tools are no longer available (see \cite{book1,book3}). As a consequence, the behavior of operator ideals in this context is markedly different and often exhibits genuinely pathological features.

In this direction, several recent works have explored lineability and spaceability beyond the Banach framework (see \cite{Pellegrino2,BCFPmaximal}). These contributions highlight the richness of the quasi-Banach setting and reveal new phenomena that cannot be obtained by classical methods. In recent years, new approaches to lineability and spaceability have been developed, providing more flexible and constructive methods to produce large linear structures (see \cite{Gustavo,BCAJ,calderon,favaro,Favaro2024}). Moreover, refined notions such as pointwise lineability and pointwise transversal lineability have been investigated in sequence spaces, revealing new structural phenomena (see \cite{transversal,PAR}). These developments highlight the growing interest in understanding linear structures beyond classical settings. More recently, the investigation of lineability and spaceability phenomena within classes of operators has gained renewed attention. For instance, large linear structures have been identified in classes of summing operators and related operator ideals (see \cite{Albuquerque}), as well as in sets of operators defined by geometric or norm-attaining properties (see \cite{Dantas}). These results reinforce the richness of the operator-theoretic perspective in the study of lineability. In particular,  Tomaz \cite{Tomaz2018} obtained a $\mathfrak{c}$-dimensional linear subspace of non-$\left(r,s\right)$-absolutely summing operators on $\ell_p$ for $0 < p < 1$ . The bounded-operator construction below is included with a detailed isometric proof, and serves as the basis for the compact-operator results and the extensions developed here.

The principal purpose of the present work is to describe closed subspaces of operators on $\ell_p$, $0<p<1$, whose nonzero elements fail \emph{every} $(r,s)$-absolute summability condition with $1\le s\le r<\infty$. For bounded operators, the disjointization argument transfers the failure of absolute summability of the identity, as established by Maddox \cite{Maddox}, to a closed subspace isometric to $\ell_p$. We preserve this construction and provide the full operator quasi-norm estimates, including uniform convergence of the defining series and closedness of the range.

The compact-operator case poses a separate difficulty. A diagonal multiplier is compact when its coefficients tend to zero, but this fact alone does not imply non-absolute summability. We overcome this difficulty with an explicit sequence which is weakly $1$-summable yet whose images under a compact diagonal operator have quasi-norms uniformly bounded below. Thus the \emph{same operator and the same sequence} work simultaneously for all $(r,s)$. We further obtain a sharp characterization of the diagonal multipliers with this property, together with a rigorous treatment of the slowly decaying logarithmic example. The Orlicz summability theorem for $\ell_1$, recalled below as in \cite[Chapter~1, p.~13]{Diestel}, provides the complementary implication in the diagonal criterion.

Finally, we use isomorphic embeddings of $\ell_p$ into arbitrary quasi-Banach spaces to transfer both the bounded and the compact constructions. The argument takes the exact disjoint-support identity \emph{before} applying the embedding; consequently, no preservation of coordinate supports in the new codomain is required. The corresponding two-sided estimates also ensure closedness. The paper has the two-section organization: Section \ref{s1} contains the introduction, notation and basic geometric construction; Section \ref{s2} contains the principal results, examples and extensions.

\medskip
\noindent\textbf{Notation and basic facts.}
Let $E$ and $F$ be Banach or quasi-Banach spaces over $\mathbb K=\mathbb R$ or $\mathbb C$. We denote by $\mathcal L(E,F)$ the vector space of bounded linear operators, endowed with the operator quasi-norm (see \cite[p. 322]{kaltonlp})
\[
\|T\|=\sup\{\|Tx\|_F:\|x\|_E\le1\}.
\]
We write $\mathcal K(E,F)$ for the compact operators. This is a closed vector subspace of $\mathcal L(E,F)$: the uniform operator quasi-norm limit of compact maps is compact, since a sufficiently fine finite net for the image of the unit ball of an approximating compact map remains an approximate finite net for the limiting map.

A quasi-norm on $E$ is a positive-definite, absolutely homogeneous function $\|\cdot\|$ such that, for some constant $C\ge1$,
\[
\|x+y\|\le C(\|x\|+\|y\|)\qquad (x,y\in E).
\]
If $E$ is complete for its induced topology, it is a quasi-Banach space. A quasi-norm is a $p$-norm, $0<p\le1$, if
\[
\|x+y\|^p\le\|x\|^p+\|y\|^p.
\]
Throughout, we work with the $p$-Banach space $\ell_p$, $0<p<1$, and its standard $p$-norm
\[
\|x\|_p=\left(\sum_{n=1}^\infty|x_n|^p\right)^{1/p}.
\]
In particular, $\|x\|_1\le\|x\|_p$ and $\ell_p$ embeds continuously into $\ell_1$. Moreover, the \emph{topological dual} $(\ell_p)'$ identifies with $\ell_\infty$: every $f\in(\ell_p)'$ has the form (see, e.g., \cite[p. 21]{book1})
\begin{equation}\label{eq:dual}
	f(x)=\sum_{n=1}^\infty b_nx_n,\quad b\in\ell_\infty,
	\qquad |f(x)|\le\|b\|_\infty\|x\|_1.
\end{equation}
Indeed, the inequality gives continuity of each such functional; conversely, the scalar sequence $b_n=f(e_n)$ is bounded and the density of finitely supported vectors yields the representation.

For $1\le s\le r<\infty$ we denote by $\Pi_{(r,s)}(E,F)$ the set of $(r,s)$-absolutely summing operators. Thus $A\in\mathcal L(E,F)$ belongs to $\Pi_{(r,s)}(E,F)$ if
\begin{equation}\label{eq:definition}
	\sum_{m=1}^\infty\|Ax^{(m)}\|_F^r<\infty
	\quad\text{whenever}\quad
	\sum_{m=1}^\infty|f(x^{(m)})|^s<\infty\ \text{for every }f\in E'.
\end{equation}
A sequence satisfying the condition on the right is called \emph{weakly $s$-summable}. We use the standard sequential definition in the quasi-Banach framework; for the Banach case see \cite{Diestel} and, for $\ell_p$, see \cite{Maddox}. We write (for simplicity of notation)
\begin{equation}\label{eq:badset}
	\mathscr N(E,F)=\mathcal L(E,F)\setminus
	\bigcup_{1\le s\le r<\infty}\Pi_{(r,s)}(E,F).
\end{equation}
It is important that $A\in\mathscr N(E,F)$ means the failure of $(r,s)$-absolute summability for \emph{each} admissible pair, and not merely for one pair.

We recall the notions of lineability and spaceability from \cite{Aron,Bmax}. If $A$ is a subset of a vector space $V$, then $A$ is \emph{$\alpha$-lineable} if $A\cup\{0\}$ contains a linear subspace of Hamel dimension $\alpha$. If $V$ is also a topological vector space, it is \emph{$\alpha$-spaceable} if this subspace may be chosen closed. We say that $A$ is \emph{maximal spaceable} in $V$ when $\alpha=\dim V$. 

For a set $A$, we denote its cardinality by $|A|$, and we write  $\mathfrak c:=|\mathbb R|$ for the cardinality of the continuum. The dimension calculation used below follows from the exact general result of Popoola and Tweddle \cite[Proposition]{Popoola}: \emph{every infinite-dimensional complete metrizable topological vector space has Hamel dimension at least $\mathfrak c$, and if it is separable then its dimension is precisely $\mathfrak c$.} In our setting one obtains
\begin{equation}\label{eq:dimensions}
	\dim\ell_p=\dim\mathcal K(\ell_p,\ell_p)
	=\dim\mathcal L(\ell_p,\ell_p)=\mathfrak c.
\end{equation}
For completeness, $\ell_p$ is infinite dimensional, separable and complete, so the cited proposition gives $\dim\ell_p=\mathfrak c$. Since each operator $\ell_p\to\ell_p$ is determined by the images of the standard unit vectors, whose possibilities form a set of cardinality $\mathfrak c$, we have $|\mathcal L(\ell_p,\ell_p)|\le\mathfrak c$. On the other hand, the injection
\[
\ell_p\longrightarrow\mathcal K(\ell_p,\ell_p),\qquad
z\longmapsto\big(x\mapsto x_1z\big)
\]
shows $\dim\mathcal K(\ell_p,\ell_p)\ge\mathfrak c$; both operator spaces therefore have dimension exactly $\mathfrak c$. This also explains precisely where the Popoola--Tweddle dimension theorem is used rather than attributing the operator-space dimension to it without the accompanying cardinality argument.

\begin{theorem}[Maddox {\cite[Theorem 4]{Maddox}}]\label{t1}
	Let $0<p<1$ and $1\le s\le r<\infty$. Then the identity operator $i:\ell_p\longrightarrow\ell_p$
	is not $(r,s)$-absolutely summing.
\end{theorem}

\noindent In particular, for every fixed admissible pair $(r,s)$ there exists a weakly $s$-summable sequence $(x^{(m)})\subset\ell_p$ such that
\begin{equation}\label{eq:maddoxwitness}
	\sum_{m=1}^\infty\|x^{(m)}\|_p^r=\infty.
\end{equation}
This sequence may depend on the pair $(r,s)$; a common sequence will be constructed later for the compact result. The following form of the classical Orlicz theorem will be used only in the diagonal characterization; see \cite[Chapter~1, p.~13]{Diestel}.
\begin{theorem}[Orlicz]\label{thm:orlicz}
	For every weakly $1$-summable sequence $(z^{(m)})_{m\ge1}\subset\ell_1$ one has
	\[
	\sum_{m=1}^\infty\|z^{(m)}\|_1^2<\infty.
	\]
	Equivalently, $\operatorname{id}_{\ell_1}\in\Pi_{(2,1)}(\ell_1,\ell_1)$.
\end{theorem}

\medskip
\noindent\textbf{Disjoint-support operators.}
Let $(\mathbb N_k)_{k\ge1}$ be a partition of $\mathbb N$ into infinitely many pairwise disjoint infinite sets (see, e.g., \cite[p. 127]{book}), and write
\[
\mathbb N_k=\{n_1^{(k)}<n_2^{(k)}<\cdots\},\qquad
\ell_p^{(k)}=\{x\in\ell_p:x_n=0\text{ if }n\notin\mathbb N_k\}.
\]
For every $k$, define $u^{(k)}:\ell_p\to\ell_p^{(k)}$ by
\[
\big(u^{(k)}x\big)_{n_m^{(k)}}=x_m\quad(m\ge1)
\]
and let $i_k:\ell_p^{(k)}\hookrightarrow\ell_p$ be the canonical inclusion. We consider 
\begin{equation}\label{eq:uk}
	u_k=i_k\circ u^{(k)}:\ell_p\longrightarrow\ell_p,
	\qquad u_kx=\sum_{m=1}^\infty x_m e_{n_m^{(k)}}.
\end{equation}
For any $x\in\ell_p$, since the correspondence $m\mapsto n_m^{(k)}$ is injective and exhausts $\mathbb N_k$, we have
\begin{equation}\label{eq:ukisometry}
	\|u_kx\|_p^p
	=\sum_{n\in\mathbb N_k}|(u_kx)_n|^p
	=\sum_{m=1}^\infty|x_m|^p=\|x\|_p^p.
\end{equation}
Thus $u_k$ is a linear isometry and its range is $\ell_p^{(k)}$. If $k\ne j$, the supports of $u_kx$ and $u_jy$ are disjoint for \emph{any} $x,y\in\ell_p$. Consequently, for each finite set $F\subset\mathbb N$ and scalars $(a_j)_{j\in F}$,
\begin{equation}\label{eq:finiteidentity}
	\left\|\sum_{j\in F}a_ju_jx\right\|_p^p
	=\sum_{j\in F}\sum_{n\in\mathbb N_j}|a_j(u_jx)_n|^p
	=\sum_{j\in F}|a_j|^p\sum_{m=1}^\infty|x_m|^p
	=\left(\sum_{j\in F}|a_j|^p\right)\|x\|_p^p.
\end{equation}
This \emph{exact equality}, rather than a quasi-triangle inequality, is the foundation of the following isometric constructions. The operators $(u_k)_{k\ge1}$ are linearly independent: if $\sum_{j\in F}a_ju_j=0$, apply~\eqref{eq:finiteidentity} to $x=e_1$ to obtain $\sum_{j\in F}|a_j|^p=0$.

\section{Main results}\label{s2}

\begin{theorem}\label{thm3.1}
	For every $0<p<1$, the set
	\[
	\mathcal L(\ell_p;\ell_p)\setminus
	\bigcup_{1\le s\le r<\infty}\Pi_{(r,s)}(\ell_p;\ell_p)
	\]
	is maximal spaceable. More precisely, it contains, except for the zero operator, a closed subspace of Hamel dimension $\mathfrak c$ which is \emph{linearly isometric to $\ell_p$}.
\end{theorem}

\begin{proof}
	We proceed in five steps, explicitly verifying all convergence and isometry assertions.
	
	\smallskip\noindent\emph{Step 1: Convergence in the operator quasi-norm.}
	Fix $a=(a_j)\in\ell_p$. For $N\in\mathbb N$, put
	\[
	S_N(a)=\sum_{j=1}^{N}a_ju_j\in\mathcal L(\ell_p;\ell_p).
	\]
	Applying~\eqref{eq:finiteidentity} with $F=\{N+1,\dots,M\}$, we obtain for every $M>N$ and $x\in\ell_p$,
	\[
	\|(S_M(a)-S_N(a))x\|_p^p
	=\left(\sum_{j=N+1}^{M}|a_j|^p\right)\|x\|_p^p.
	\]
	Taking $p$th roots and then the supremum over $\|x\|_p\le1$ gives
	\begin{equation}\label{eq:tails}
		\|S_M(a)-S_N(a)\|_{\mathcal L(\ell_p;\ell_p)}
		=\left(\sum_{j=N+1}^{M}|a_j|^p\right)^{1/p}.
	\end{equation}
	To see that this is an \emph{equality}, not merely an upper bound, substitute $x=e_1$, which has unit $p$-norm. The right-hand side tends to zero as $M,N\to\infty$. Because $\mathcal L(\ell_p;\ell_p)$ is complete in the operator quasi-norm, $(S_N(a))$ converges to a bounded linear operator, denoted
	\begin{equation}\label{eq:S}
		S(a)=\sum_{j=1}^\infty a_ju_j.
	\end{equation}
	This also proves that the infinite operator series converges uniformly on the unit ball of $\ell_p$.
	
	\smallskip\noindent\emph{Step 2: An exact pointwise norm identity.}
	For each fixed $x\in\ell_p$, the operator convergence implies $S_N(a)x\to S(a)x$ in $\ell_p$. From~\eqref{eq:finiteidentity},
	\[
	\|S_N(a)x\|_p^p
	=\left(\sum_{j=1}^{N}|a_j|^p\right)\|x\|_p^p.
	\]
	The quasi-norm is continuous, and the sum on the right tends to $\|a\|_p^p$. Passing to the limit therefore yields
	\begin{equation}\label{eq:Sidentity}
		\boxed{\ \|S(a)x\|_p^p=\|a\|_p^p\|x\|_p^p\qquad(a,x\in\ell_p).\ }
	\end{equation}
	In particular, for each $x$, the vector $S(a)x$ is obtained from the absolutely coordinatewise disjoint family $(a_ju_jx)$ with no cancellation between blocks.
	
	\smallskip\noindent\emph{Step 3: Linear isometry and closedness.}
	Linearity of $S$ follows from linearity of all partial sums and continuity of the limit. Equation~\eqref{eq:Sidentity} gives for $\|x\|_p\le1$,
	\[
	\|S(a)x\|_p=\|a\|_p\|x\|_p\le\|a\|_p.
	\]
	Taking the supremum, and once again evaluating at $x=e_1$, we obtain both inequalities needed for the norm equality:
	\begin{equation}\label{eq:Sisometry}
		\boxed{\ \|S(a)\|_{\mathcal L(\ell_p;\ell_p)}=\|a\|_p.\ }
	\end{equation}
	In particular, if $S(a)=0$, then $\|a\|_p=0$, so $a=0$. Moreover, for $a,b\in\ell_p$,
	\begin{equation}\label{eq:Sdifference}
		\|S(a)-S(b)\|=\|S(a-b)\|=\|a-b\|_p.
	\end{equation}
	Let $(S(a^{(m)}))$ converge to some $A\in\mathcal L(\ell_p;\ell_p)$. It is Cauchy, and by~\eqref{eq:Sdifference} the coefficients $(a^{(m)})$ form a Cauchy sequence in $\ell_p$. Completeness supplies $a\in\ell_p$ with $a^{(m)}\to a$. By the same isometry, $S(a^{(m)})\to S(a)$, so $A=S(a)$. Thus $S(\ell_p)$ is a \emph{closed} linear subspace, not just a dense subspace of its closure.
	
	\smallskip\noindent\emph{Step 4: Failure of all absolute summability conditions.}
	Fix $a\in\ell_p\setminus\{0\}$ and an arbitrary pair $1\le s\le r<\infty$. By Maddox's theorem (Theorem~\ref{t1}) there exists a weakly $s$-summable sequence $(x^{(m)})\subset\ell_p$ satisfying~\eqref{eq:maddoxwitness}. Applying~\eqref{eq:Sidentity} to each $x^{(m)}$ yields
	\[
	\|S(a)x^{(m)}\|_p=\|a\|_p\|x^{(m)}\|_p.
	\]
	Since $\|a\|_p>0$ and $r$ is finite,
	\begin{align*}
		\sum_{m=1}^\infty\|S(a)x^{(m)}\|_p^r
		&=\sum_{m=1}^\infty\big(\|a\|_p\|x^{(m)}\|_p\big)^r\\
		&=\|a\|_p^r\sum_{m=1}^\infty\|x^{(m)}\|_p^r=\infty.
	\end{align*}
	Hence $S(a)\notin\Pi_{(r,s)}(\ell_p;\ell_p)$. The pair was arbitrary and $S(a)$ is independent of its choice. Therefore
	\[
	S(\ell_p)\setminus\{0\}\subset\mathscr N(\ell_p,\ell_p).
	\]
	Notice that the choice of a witness sequence may depend on $(r,s)$ here; no unjustified assumption of a common Maddox witness has been made.
	
	\smallskip\noindent\emph{Step 5: Dimension and maximality.}
	Because $S$ is injective, $\dim S(\ell_p)=\dim\ell_p=\mathfrak c$. The ambient dimension identity~\eqref{eq:dimensions} gives $\dim\mathcal L(\ell_p;\ell_p)=\mathfrak c$. The closed subspace $S(\ell_p)$ thus has full Hamel dimension, completing the maximal spaceability proof.
\end{proof}
This result extends and complements the results established in \cite{Tomaz2018}.
\begin{corollary}\label{cor:ordinary}
	For every $0<p<1$, the set $\mathcal L(\ell_p;\ell_p)\setminus\bigcup_{1\le r<\infty}\Pi_{(r,r)}(\ell_p;\ell_p)$ is maximal spaceable.
\end{corollary}
\begin{proof}
	By definition, the union over $(r,r)$ is contained in the union over $1\le s\le r<\infty$; hence the closed subspace in Theorem~\ref{thm3.1} is also contained in this complement, apart from zero. It has ambient maximal dimension by~\eqref{eq:dimensions}.
\end{proof}

\begin{remark}\label{rem:method}
	The isometric structure of the above proof is crucial. Since the $u_k$ have disjoint target supports, the $p$th power of the quasi-norm splits exactly, and the operator quasi-norm of the series is determined by the $\ell_p$-norm of its coefficients. In particular, closedness is a consequence of the isometry~\eqref{eq:Sdifference}; it cannot be deduced merely from the coordinatewise convergence of coefficients. This construction is related to methods used in quasi-Banach sequence spaces \cite{Pellegrino2,Tomaz2018}, and does not require local convexity or operator-range spaceability criteria such as \cite{KT}.
\end{remark}

\begin{lemma}[A common witness]\label{lem:witness}
	Let $0<p<1$, let $X$ be a quasi-Banach space, and let $A\in\mathcal L(\ell_p,X)$. Suppose that
	\begin{equation}\label{eq:unboundedl1}
		\sup\{\|Ax\|_X:x\in\ell_p,\ \|x\|_1\le1\}=\infty.
	\end{equation}
	Then there exists a sequence $(x^{(k)})_{k\ge1}\subset\ell_p$ such that
	\begin{equation}\label{eq:common}
		\sum_{k=1}^\infty\|x^{(k)}\|_1<\infty,
		\qquad\|Ax^{(k)}\|_X\ge1\quad\text{for every }k.
	\end{equation}
	Consequently $A$ fails to be $(r,s)$-absolutely summing for \emph{every} $1\le s\le r<\infty$, with the same sequence witnessing all failures.
\end{lemma}
\begin{proof}
	For each $k$, choose $y^{(k)}\in\ell_p$ with $\|y^{(k)}\|_1\le1$ and $\|Ay^{(k)}\|_X\ge2^k$, possible by~\eqref{eq:unboundedl1}. Define $x^{(k)}=2^{-k}y^{(k)}$. Then $\|x^{(k)}\|_1\le2^{-k}$ and $\|Ax^{(k)}\|_X\ge1$. For every $f\in(\ell_p)'$, formula~\eqref{eq:dual} yields
	\[
	\sum_k|f(x^{(k)})|\le\|b\|_\infty\sum_k\|x^{(k)}\|_1\le\|b\|_\infty<\infty.
	\]
	Therefore $(x^{(k)})$ is weakly $1$-summable. A scalar summable sequence belongs to $\ell_s$ for every $s\ge1$, so $(x^{(k)})$ is also weakly $s$-summable for all such $s$. On the other hand, for every finite $r\ge1$,
	\[
	\sum_k\|Ax^{(k)}\|_X^r\ge\sum_k1=\infty.
	\]
	No choice in this argument depends on $r$ or $s$.
\end{proof}

\begin{proposition}[A sharp diagonal criterion]\label{prop:diag}
	Let $0<p<1$, put $q=p/(1-p)$, and let $\lambda=(\lambda_n)\in\ell_\infty$. For the diagonal operator
	\[
	D_\lambda:\ell_p\longrightarrow\ell_p,\qquad
	D_\lambda(x_n)=(\lambda_nx_n),
	\]
	the following conditions are equivalent:
	\begin{enumerate}
		\item[(i)] $\lambda\notin\ell_q$;
		\item[(ii)] $\sup\{\|D_\lambda x\|_p:x\in\ell_p,\ \|x\|_1\le1\}=\infty$;
		\item[(iii)] $D_\lambda\notin\bigcup_{1\le s\le r<\infty}\Pi_{(r,s)}(\ell_p;\ell_p)$.
	\end{enumerate}
	If $\lambda\in\ell_q$, then $D_\lambda\in\Pi_{(2,1)}(\ell_p;\ell_p)$. Also, $D_\lambda$ is compact if and only if $\lambda\in c_0$.
\end{proposition}
\begin{proof}
	First, $\lambda\in\ell_\infty$ guarantees boundedness, and for every $x\in\ell_p$,
	\[
	\|D_\lambda x\|_p^p=\sum_n|\lambda_n|^p|x_n|^p
	\le\|\lambda\|_\infty^p\|x\|_p^p.
	\]
	Equality for the operator quasi-norm is obtained by evaluating at $e_n$ and taking the supremum over $n$, so $\|D_\lambda\|=\|\lambda\|_\infty$.
	
	Suppose (i). For any nonempty finite $F\subset\mathbb N$, set $H_F=\sum_{n\in F}|\lambda_n|^q$. Whenever $H_F>0$, define a finitely supported vector $z^F$ by
	\[
	z_n^F=\begin{cases}|\lambda_n|^q/H_F,&n\in F,\\0,&n\notin F.\end{cases}
	\]
	Then $\|z^F\|_1=1$. The algebraic identity $p(q+1)=q$ gives
	\begin{align*}
		\|D_\lambda z^F\|_p^p
		&=\sum_{n\in F}|\lambda_n|^p\frac{|\lambda_n|^{pq}}{H_F^p}
		=\frac{1}{H_F^p}\sum_{n\in F}|\lambda_n|^{p(q+1)}\\
		&=\frac{1}{H_F^p}\sum_{n\in F}|\lambda_n|^q
		=H_F^{1-p}.
	\end{align*}
	Since $(1-p)/p=1/q$, this yields
	\begin{equation}\label{eq:H}
		\|D_\lambda z^F\|_p=H_F^{1/q}.
	\end{equation}
	If $\lambda\notin\ell_q$, the values $H_F$ are arbitrarily large. Hence (ii) holds. The common-witness Lemma~\ref{lem:witness} implies (iii).
	
	Conversely, suppose $\lambda\in\ell_q$. Let $x\in\ell_1$ and apply H\"older's inequality with the conjugate exponents $1/(1-p)$ and $1/p$ to the nonnegative sequences $|\lambda_n|^p$ and $|x_n|^p$. We obtain
	\begin{align*}
		\|D_\lambda x\|_p^p
		&=\sum_n|\lambda_n|^p|x_n|^p\\
		&\le\left(\sum_n|\lambda_n|^{p/(1-p)}\right)^{1-p}
		\left(\sum_n|x_n|\right)^p
		=\|\lambda\|_q^p\|x\|_1^p.
	\end{align*}
	Thus $D_\lambda$ extends to a bounded linear operator $B:\ell_1\to\ell_p$ with $\|B\|\le\|\lambda\|_q$, and on $\ell_p$ we have $D_\lambda=B\circ\iota$, where $\iota:\ell_p\hookrightarrow\ell_1$ denotes the canonical inclusion.
	
	Let $(x^{(m)})\subset\ell_p$ be any weakly $1$-summable sequence. Every continuous functional on $\ell_1$ restricts along $\iota$ to a continuous functional on $\ell_p$; hence $(\iota x^{(m)})$ is weakly $1$-summable in $\ell_1$. By Orlicz's theorem (Theorem~\ref{thm:orlicz}),
	\[
	\sum_m\|\iota x^{(m)}\|_1^2<\infty.
	\]
	Consequently,
	\[
	\sum_m\|D_\lambda x^{(m)}\|_p^2
	=\sum_m\|B(\iota x^{(m)})\|_p^2
	\le\|B\|^2\sum_m\|\iota x^{(m)}\|_1^2<\infty.
	\]
	This proves $D_\lambda\in\Pi_{(2,1)}(\ell_p;\ell_p)$ and contradicts (iii); hence (iii) implies (i). It also proves directly that (ii) fails when $\lambda\in\ell_q$.
	
	For compactness, let $D_{\lambda,N}$ be the operator obtained by retaining only the first $N$ diagonal entries. It has finite rank, and from the exact diagonal norm identity,
	\[
	\|D_\lambda-D_{\lambda,N}\|=\sup_{n>N}|\lambda_n|.
	\]
	Thus $\lambda\in c_0$ implies convergence in the operator quasi-norm to finite-rank operators, hence compactness. Conversely, if $D_\lambda$ is compact, the sequence $(D_\lambda e_n)=(\lambda_ne_n)$ must have a convergent subsequence whenever the coefficients stay bounded away from zero. But if infinitely many $|\lambda_n|\ge\varepsilon>0$, then for distinct such indices $n,m$,
	\[
	\|\lambda_ne_n-\lambda_me_m\|_p^p=|\lambda_n|^p+|\lambda_m|^p\ge2\varepsilon^p,
	\]
	so no subsequence can converge, a contradiction. Therefore $\lambda_n\to0$.
\end{proof}

\begin{corollary}[The logarithmic compact example]\label{cor:log}
	For every $0<p<1$, let $\lambda_n=(\log(n+1))^{-1/p}$. Then the diagonal multiplier $D_\lambda$ is compact and satisfies
	\[
	D_\lambda\notin\Pi_{(r,s)}(\ell_p;\ell_p)
	\quad\text{for all } \ 1\le s\le r<\infty.
	\]
\end{corollary}
\begin{proof}
	Since $\lambda_n\to0$, the multiplier is compact. For $q=p/(1-p)$,
	\[
	\sum_{n=1}^\infty|\lambda_n|^q
	=\sum_{n=1}^\infty\frac{1}{(\log(n+1))^{1/(1-p)}}.
	\]
	For sufficiently large $n$, $(\log(n+1))^{1/(1-p)}\le n$, so the terms eventually dominate $1/n$. Their sum therefore diverges, and Proposition~\ref{prop:diag} proves the result.
\end{proof}

\begin{lemma}[An explicit compact operator and a single witness]\label{lem4.1}
	Let $0<p<1$ and $q=p/(1-p)$. Put $w_n=n^{-1/q}$ and define
	\begin{equation}\label{eq:D}
		K:\ell_p\longrightarrow\ell_p,\qquad K(x_n)=(w_nx_n)_{n\ge1}.
	\end{equation}
	Then $K$ is compact, has operator quasi-norm $\|K\|=1$, and belongs to $\mathscr N(\ell_p,\ell_p)$. In fact, there exist pairwise disjoint finite intervals $I_k\subset\mathbb N$ and $x^{(k)}\in\ell_p$ supported in $I_k$ such that
	\begin{equation}\label{eq:explicit}
		\sum_{k=1}^\infty\|x^{(k)}\|_1=1,
		\qquad\|Kx^{(k)}\|_p\ge1\quad(k\ge1).
	\end{equation}
	The same sequence $(x^{(k)})$ works for every pair $(r,s)$.
\end{lemma}
\begin{proof}
	The scalar weights satisfy $0<w_n\le1$, $w_1=1$ and $w_n\to0$. As in the preceding proposition, the finite-rank truncations $K_N$ satisfy
	\[
	\|K-K_N\|=\sup_{n>N}n^{-1/q}=(N+1)^{-1/q}\longrightarrow0.
	\]
	Thus $K$ is compact, and the equality $\|K\|=1$ follows from evaluating at $e_1$.
	
	The harmonic series diverges on every tail. We may therefore inductively choose disjoint consecutive finite intervals $I_k$, each to the right of the preceding one, such that
	\[
	H_k:=\sum_{n\in I_k}\frac1n\ge2^{kq}.
	\]
	Define $x^{(k)}=(x_n^{(k)})$ by
	\begin{equation}\label{eq:xk}
		x_n^{(k)}=\begin{cases}
			\displaystyle\frac{2^{-k}}{H_kn},&n\in I_k,\\[4pt]
			0,&n\notin I_k.
		\end{cases}
	\end{equation}
	Each vector has finite support, so it belongs to $\ell_p$, and
	\[
	\|x^{(k)}\|_1=\sum_{n\in I_k}\frac{2^{-k}}{H_kn}
	=\frac{2^{-k}}{H_k}H_k=2^{-k}.
	\]
	Next, $1/q=(1-p)/p$, so $p+p/q=1$. Therefore
	\begin{align*}
		\|Kx^{(k)}\|_p^p
		&=\sum_{n\in I_k} n^{-p/q}
		\left(\frac{2^{-k}}{H_kn}\right)^p\\
		&=\frac{2^{-kp}}{H_k^p}
		\sum_{n\in I_k}n^{-p(1+1/q)}\\
		&=\frac{2^{-kp}}{H_k^p}\sum_{n\in I_k}\frac1n
		=2^{-kp}H_k^{1-p}.
	\end{align*}
	Taking $p$th roots and using $H_k\ge2^{kq}$ yields
	\[
	\|Kx^{(k)}\|_p
	=2^{-k}H_k^{(1-p)/p}
	=2^{-k}H_k^{1/q}\ge2^{-k}2^k=1.
	\]
	This proves~\eqref{eq:explicit}. If $f\in(\ell_p)'$, by~\eqref{eq:dual},
	\[
	\sum_{k=1}^\infty|f(x^{(k)})|
	\le\|b\|_\infty\sum_{k=1}^\infty\|x^{(k)}\|_1
	=\|b\|_\infty<\infty.
	\]
	Consequently $(x^{(k)})$ is weakly $1$-summable and hence weakly $s$-summable for each $s\ge1$. In contrast, for every finite $r\ge1$,
	\[
	\sum_{k=1}^\infty\|Kx^{(k)}\|_p^r\ge\sum_{k=1}^\infty1=\infty.
	\]
	The simultaneous non-summability follows for all $1\le s\le r<\infty$.
\end{proof}

\begin{theorem}[Maximal spaceability of compact operators]\label{thm:compact}
	Let $0<p<1$. Then
	\[
	\mathcal K(\ell_p,\ell_p)\setminus
	\bigcup_{1\le s\le r<\infty}\Pi_{(r,s)}(\ell_p;\ell_p)
	\]
	is maximal spaceable in $\mathcal K(\ell_p,\ell_p)$ and contains, apart from zero, a closed subspace linearly isometric to $\ell_p$. This subspace has dimension $\mathfrak c$, so it is also of maximal dimension in $\mathcal L(\ell_p,\ell_p)$. Every nonzero operator in it has infinite rank, and the same weakly $1$-summable sequence witnesses the failure of every admissible $(r,s)$-condition for every such operator.
\end{theorem}
\begin{proof}
	Let $K$ and $(x^{(k)})$ be as in Lemma~\ref{lem4.1}. Set
	\[
	S_K(a)=S(a)\circ K=\sum_{j=1}^\infty a_j(u_j\circ K),
	\qquad a=(a_j)\in\ell_p,
	\]
	where the identity between the two expressions follows from the operator quasi-norm convergence of the series defining $S(a)$.
	
	\smallskip\noindent\emph{Compactness of the whole operator series.}
	Since $K$ is compact and every $u_j$ is bounded, $u_jK$ is compact. For any $M>N$, equation~\eqref{eq:finiteidentity} applied to the vector $Kx$ yields
	\[
	\left\|\sum_{j=N+1}^{M}a_ju_jKx\right\|_p^p
	=\left(\sum_{j=N+1}^{M}|a_j|^p\right)\|Kx\|_p^p.
	\]
	Taking the supremum over the unit ball, we obtain the exact tail operator quasi-norm
	\[
	\left\|\sum_{j=N+1}^{M}a_ju_jK\right\|
	=\left(\sum_{j=N+1}^{M}|a_j|^p\right)^{1/p}\|K\|.
	\]
	Thus the partial sums form a Cauchy sequence of \emph{compact operators}. Since $\mathcal K(\ell_p,\ell_p)$ is closed in $\mathcal L(\ell_p,\ell_p)$, their limit $S_K(a)=S(a)K$ is compact.
	
	\smallskip\noindent\emph{The second exact isometry.}
	Applying~\eqref{eq:Sidentity} to $y=Kx$ gives, for every $x\in\ell_p$,
	\begin{equation}\label{eq:compactidentity}
		\|S_K(a)x\|_p^p
		=\|S(a)(Kx)\|_p^p
		=\|a\|_p^p\|Kx\|_p^p.
	\end{equation}
	Consequently,
	\[
	\|S_K(a)x\|_p\le\|a\|_p\|K\|\|x\|_p
	=\|a\|_p\|x\|_p.
	\]
	Since $Ke_1=e_1$ and $\|e_1\|_p=1$, the lower bound is attained at $x=e_1$:
	\[
	\|S_K(a)e_1\|_p=\|a\|_p\|Ke_1\|_p=\|a\|_p.
	\]
	Thus
	\begin{equation}\label{eq:compisometry}
		\boxed{\ \|S_K(a)\|_{\mathcal K(\ell_p,\ell_p)}=\|a\|_p.\ }
	\end{equation}
	Linearity follows immediately from the construction. For $a,b\in\ell_p$,
	\[
	\|S_K(a)-S_K(b)\|
	=\|S_K(a-b)\|=\|a-b\|_p.
	\]
	Therefore, if $S_K(a^{(m)})$ is Cauchy or convergent in $\mathcal K(\ell_p,\ell_p)$, then $(a^{(m)})$ is Cauchy in $\ell_p$, has a limit $a$, and $S_K(a^{(m)})\to S_K(a)$. This proves that $S_K(\ell_p)$ is closed and isometric to $\ell_p$.
	
	\smallskip\noindent\emph{One witness for all pairs and all nonzero operators.}
	Fix $a\ne0$. By \eqref{eq:explicit} and \eqref{eq:compactidentity},
	\[
	\|S_K(a)x^{(k)}\|_p
	=\|a\|_p\|Kx^{(k)}\|_p\ge\|a\|_p>0
	\qquad(k\ge1).
	\]
	The sequence $(x^{(k)})$ is weakly $1$-summable and hence weakly $s$-summable for all $s\ge1$. Therefore, for every $1\le s\le r<\infty$,
	\[
	\sum_{k=1}^\infty\|S_K(a)x^{(k)}\|_p^r
	\ge\|a\|_p^r\sum_{k=1}^\infty1=\infty.
	\]
	The same sequence works simultaneously for every admissible pair, with no dependence on $a$ except through the positive factor $\|a\|_p$. Moreover, $K$ is injective and has infinite-dimensional range (all its weights are positive), while $S(a)$ is injective whenever $a\ne0$ by~\eqref{eq:Sidentity}; hence $S_K(a)=S(a)K$ has infinite rank. Finally, the isometry and~\eqref{eq:dimensions} give
	\[
	\dim S_K(\ell_p)=\mathfrak c
	=\dim\mathcal K(\ell_p,\ell_p),
	\]
	proving maximality.
\end{proof}

\begin{remark}\label{rem:compact}
	The logarithmic multiplier in Corollary~\ref{cor:log} could also be used instead of $K$ in the proof of Theorem~\ref{thm:compact}. Writing $S_{D_\lambda}(a):=S(a)\circ D_\lambda$ for this alternative construction, the exact operator quasi-norm is $\|S_{D_\lambda}(a)\|=\|a\|_p\|D_\lambda\|$. Dividing by the nonzero constant $\|D_\lambda\|$ yields an isometric embedding. The power-weight choice~\eqref{eq:D} has the advantage that its norm equals $1$ and the common witness can be given by the explicit formula~\eqref{eq:xk}. The compact examples are operator quasi-norm limits of finite-rank maps, but this does not imply they are absolutely summing: the summing classes in question need not be closed for the ambient operator quasi-norm.
\end{remark}

\begin{theorem}[Isomorphic extensions to quasi-Banach ranges]\label{thm:extension}
	Let $0<p<1$ and let $X$ be a quasi-Banach space containing an isomorphic copy of $\ell_p$. Then
	\[
	\mathcal L(\ell_p;X)\setminus
	\bigcup_{1\le s\le r<\infty}\Pi_{(r,s)}(\ell_p;X)
	\]
	and
	\[
	\mathcal K(\ell_p,X)\setminus
	\bigcup_{1\le s\le r<\infty}\Pi_{(r,s)}(\ell_p;X)
	\]
	are $\mathfrak c$-spaceable in their respective ambient spaces. Both contain closed subspaces isomorphic to $\ell_p$. If $X$ is separable, both conclusions are maximal spaceable.
\end{theorem}
\begin{proof}
	Let $T:\ell_p\to X$ be an isomorphic embedding. By definition, there are constants $0<c\le C<\infty$ such that
	\begin{equation}\label{eq:TT}
		c\|y\|_p\le\|Ty\|_X\le C\|y\|_p
		\qquad(y\in\ell_p).
	\end{equation}
	Let $S(a)$ be the operator from~\eqref{eq:S} and $K$ the compact diagonal operator in~\eqref{eq:D}. Define
	\[
	\Phi(a)=T\circ S(a)\in\mathcal L(\ell_p;X),
	\qquad\Psi(a)=T\circ S(a)\circ K\in\mathcal K(\ell_p,X).
	\]
	Both maps are linear in $a$, and $\Psi(a)$ is compact because $K$ is compact and the remaining factors are bounded.
	
	\smallskip\noindent\emph{Uniform convergence and two-sided operator bounds.}
	For the bounded family, the partial sums
	$\Phi_N(a)=\sum_{j=1}^Na_jTu_j$ converge in the operator quasi-norm because
	\[
	\|\Phi_M(a)-\Phi_N(a)\|
	\le C\left(\sum_{j=N+1}^M|a_j|^p\right)^{1/p}.
	\]
	Similarly, $\Psi_N(a)=\sum_{j=1}^Na_jTu_jK$ converge in the operator quasi-norm and are compact. Most importantly, the inequalities~\eqref{eq:TT} are applied to the \emph{entire} vector $S(a)x$ before any attempt to estimate individual summands. From~\eqref{eq:Sidentity} we obtain
	\[
	c\|a\|_p\|x\|_p
	\le\|T(S(a)x)\|_X
	\le C\|a\|_p\|x\|_p.
	\]
	Taking the supremum over $\|x\|_p\le1$ and evaluating at $x=e_1$, we get
	\begin{equation}\label{eq:Phi}
		c\|a\|_p\le\|\Phi(a)\|\le C\|a\|_p.
	\end{equation}
	Applying the same argument to $S(a)(Kx)$, we find
	\[
	c\|a\|_p\|Kx\|_p
	\le\|T(S(a)(Kx))\|_X
	\le C\|a\|_p\|Kx\|_p.
	\]
	Because $\|K\|=1$ and $Ke_1=e_1$, taking the supremum and testing at $e_1$ gives
	\begin{equation}\label{eq:Psi}
		c\|a\|_p\le\|\Psi(a)\|\le C\|a\|_p.
	\end{equation}
	These are the correct estimates in arbitrary $X$: no assumption is made that $T$ preserves coordinate supports or the exact $p$-norm equality.
	
	\smallskip\noindent\emph{Closedness.}
	If $\Phi(a^{(m)})$ converges in $\mathcal L(\ell_p;X)$, it is Cauchy; the lower bound in~\eqref{eq:Phi} applied to differences gives
	\[
	\|a^{(m)}-a^{(n)}\|_p
	\le c^{-1}\|\Phi(a^{(m)})-\Phi(a^{(n)})\|.
	\]
	Thus $(a^{(m)})$ is Cauchy in $\ell_p$, and its limit $a$ satisfies $\Phi(a^{(m)})\to\Phi(a)$ by the upper bound. The range of $\Phi$ is closed. Repeating this with~\eqref{eq:Psi} proves the range of $\Psi$ is closed in $\mathcal K(\ell_p,X)$. Both are topologically isomorphic to $\ell_p$, because the two-sided inequalities also imply injectivity and bounded inverses on their ranges.
	
	\smallskip\noindent\emph{Simultaneous non-summability.}
	Use the common witness $(x^{(k)})$ of Lemma~\ref{lem4.1}. Since $\|K\|=1$ and $\|Kx^{(k)}\|_p\ge1$, one has $\|x^{(k)}\|_p\ge1$. If $a\ne0$, then
	\begin{align*}
		\|\Phi(a)x^{(k)}\|_X
		&\ge c\|S(a)x^{(k)}\|_p
		=c\|a\|_p\|x^{(k)}\|_p\ge c\|a\|_p,\\
		\|\Psi(a)x^{(k)}\|_X
		&\ge c\|S(a)(Kx^{(k)})\|_p
		=c\|a\|_p\|Kx^{(k)}\|_p\ge c\|a\|_p.
	\end{align*}
	Since the sequence is weakly $s$-summable for every $s\ge1$, the lower bounds yield divergence of the sums of the $r$th powers for every finite $r\ge1$. Thus all nonzero operators in both closed ranges belong to the asserted complements, simultaneously for all pairs.
	
	\smallskip\noindent\emph{Dimension and maximality.}
	Each range has Hamel dimension $\mathfrak c$. If $X$ is separable, then $|X|\le\mathfrak c$. An operator in $\mathcal L(\ell_p;X)$ is determined by its action on the countable set $(e_n)$, so
	\[
	|\mathcal L(\ell_p;X)|\le|X|^{\aleph_0}\le\mathfrak c^{\aleph_0}=\mathfrak c.
	\]
	Both ambient operator spaces contain the constructed $\mathfrak c$-dimensional subspaces and therefore have Hamel dimension exactly $\mathfrak c$. This establishes maximality when $X$ is separable.
\end{proof}

\begin{remark}\label{rem:extension}
	The separability hypothesis in the maximality assertion cannot be omitted without further information on the dimension of $X$. Indeed, if $\Gamma$ is an index set with $|\Gamma|>\mathfrak c$ and $X=\ell_p(\Gamma)$, then $X$ contains a copy of $\ell_p$, but its Hamel dimension exceeds $\mathfrak c$. The injection
	\[
	X\longrightarrow\mathcal K(\ell_p,X),\qquad
	z\longmapsto(x\mapsto x_1z)
	\]
	shows that both ambient operator spaces have dimension greater than $\mathfrak c$. The theorem accordingly asserts $\mathfrak c$-spaceability for all such $X$, and maximal spaceability when the dimension is known to be $\mathfrak c$. The norm estimates~\eqref{eq:Phi}--\eqref{eq:Psi} also resolve the subtlety raised by arbitrary quasi-Banach codomains: it would be incorrect to infer closedness only from coordinatewise convergence of coefficients, or to use support disjointness after applying $T$.
\end{remark}

\begin{remark}\label{rem:finite}
	For each admissible pair, every finite-rank operator $A:\ell_p\to X$ is $(r,s)$-absolutely summing: writing $Ax=\sum_{i=1}^Nf_i(x)y_i$ with continuous $f_i\in(\ell_p)'$, a weakly $s$-summable sequence has each scalar image $(f_i(x^{(k)}))_k\in\ell_s\subseteq\ell_r$, and a finite-dimensional quasi-norm estimate gives $\sum_k\|Ax^{(k)}\|_X^r<\infty$. Thus the compact operators in our constructed subspace cannot have finite rank. They are nevertheless uniform operator quasi-norm limits of finite-rank operators, which explains why compactness alone does not force absolute summability here.
\end{remark}

\begin{ack}
The author would like to thank the anonymous referee for their careful reading of the manuscript and their valuable comments and suggestions, which helped improve the presentation of this work.
\end{ack}

\section*{Declarations}
\textbf{Conflict of Interest} The author declares no conflict of interest.

\section*{Funding}
\textbf{Financial support}	This research has no financial support.

\section*{Data Availability Statements}
No datasets were generated or analyzed during the current study.

\end{document}